\documentclass{mpi2015-cscpreprint}

\usepackage[american]{babel}
\usepackage{graphicx,subfig,float,pgfplots}
\usepackage{amssymb,amsthm,amsmath}
\usepackage{multicol}
\usepackage{makecell,tabularx}
\usepackage{relsize}

\newtheorem{theorem}{Theorem}[section]

\newtheorem{lemma}[theorem]{Lemma}
\newtheorem{proposition}{Proposition}

\newtheorem{remark}{Remark}[section]

\newtheorem*{theorem*}{Theorem}

\numberwithin{equation}{section}

\begin{document}

\title{Exponential Decay \\for a Boundary-Controlled \\Nonlinear Parabolic Reactor Model}
  
\author[$1$]{Y.~Yevgenieva}
\affil[$1$]{Max Planck Institute for Dynamics of Complex Technical Systems,
              39106 Magdeburg, Germany\\ \newline
              Institute of Applied Mathematics and Mechanics, National Academy of Sciences of Ukraine, \newline 84116 Sloviansk, Ukraine\authorcr
  \email{yevgenieva@mpi-magdeburg.mpg.de}, \orcid{0000-0002-1867-3698}}
  
\author[$2$]{A.~Zuyev}
\affil[$2$]{Max Planck Institute for Dynamics of Complex Technical Systems,
              39106 Magdeburg, Germany\\ \newline
              Institute of Applied Mathematics and Mechanics, National Academy of Sciences of Ukraine, \newline 84116 Sloviansk, Ukraine\authorcr
  \email{zuyev@mpi-magdeburg.mpg.de}, \orcid{0000-0002-7610-5621}}

\author[$2$]{C.~Prieur}
\affil[$2$]{Universit\'{e} Grenoble Alpes, CNRS, Grenoble-INP, GIPSA-lab, F-38000 Grenoble, France\authorcr
  \email{christophe.prieur@gipsa-lab.fr}, \orcid{0000-0002-4456-2019}}

\author[$3$]{P.~Benner}
\affil[$3$]{Max Planck Institute for Dynamics of Complex Technical Systems,
              39106 Magdeburg, Germany\authorcr
  \email{benner@mpi-magdeburg.mpg.de}, \orcid{0000-0003-3362-4103}}
  
\shorttitle{Exponential Decay...}
\shortauthor{Y.~Yevgenieva, A.~Zuyev, C.~Prieur, P.~Benner}
\shortdate{}
  

  
\abstract{We study an axial dispersion tubular reactor model governed by a nonlinear parabolic equation with Robin-type boundary conditions and boundary feedback control. We derive sufficient conditions for the exponential stability of the steady-state solution of the closed-loop system and provide an explicit estimate of the decay rate. In addition, numerical simulations are presented to illustrate the sharpness of the obtained decay rate for different choices of the feedback gain parameter.}

\novelty{The key contribution of our work is the following:
\begin{itemize}
    \item { Design of a stabilizing boundary control ensuring a state constraint for the axial dispersion tubular reactor by establishing the invariance of a region in the state space under the nonlinear dynamics. See Theorem \ref{th2} and Remark \ref{rem:invariance} below for more details.}
    
    \item { A sharp estimate of the exponential decay rate for solutions of the controlled reactor model is derived.}
    
    \item Numerical simulations are presented to illustrate the sharpness of the obtained theoretical stability estimates.
\end{itemize}}

\maketitle

\section{Introduction}\label{sec:intro}

This paper is dedicated to the stability analysis and the boundary stabilization of a tubular reactor, that is usually modeled as partial differential equation due to space-dependence of its dynamics. {While stability analysis via linearization is by now well established, relying in particular on semigroup theory and control methods for distributed parameter systems~\cite{HasWinDoc2020,WinDocLig2000}, the asymptotic analysis of the nonlinear system requires a more delicate treatment. In this work, we focus on the stabilization of the nonlinear system itself, without resorting to linearization. Nonlinear effects may substantially alter the closed-loop dynamics. In particular, they may prevent global stabilization and lead only to local asymptotic stability with a bounded region of attraction (cf.~\cite{lhachemi2023local}).}

An exponential decay rate estimate for a dispersed flow reactor model with saturated nonlinearity was obtained in our previous paper~\cite{YevZB2025} using Lyapunov's direct method. There it was shown by numerical simulations, that the derived theoretical estimate is not sharp. Consequently, the problem of constructing sharp theoretical decay estimates for this class of parabolic systems with boundary control has remained open. The present paper aims to fill this gap while also considering a more realistic model of the nonlinear reaction rate without saturation.

The literature on the stability of such parabolic nonlinear dynamics is not immense. In the paper~\cite{fursikov2001stabilizability}, the stabilizability of a quasilinear parabolic equation by means of boundary feedback control was studied, absolute stability conditions for semilinear parabolic PDEs under boundary control were derived in~\cite{hagen2006absolute}.
More recently, significant advances have been achieved through backstepping and Lyapunov-based techniques~(see, e.g., \cite{karafyllis2019small} and references therein).
Although these works establish powerful methodologies for boundary stabilization of linear, semilinear, and quasilinear parabolic systems, only limited attention has been devoted to nonlinear reaction models with physically motivated Robin-type boundary feedback laws.
The paper~\cite{lhachemi2020boundary} treats a linear reaction–diffusion model with Robin boundary conditions and delayed states, and although it proves exponential stability of the closed-loop system, it does not derive explicit analytical decay-rate formulas in terms of the physical PDE parameters.

The paper \cite{lhachemi2026boundary} considers a nonlinear parabolic equation with a nonlinear diffusion, and sliding-mode technique for the observation of the reaction-diffusion with a nonlinear term are developped in~\cite{mohet2023infinite}. As in \cite{HasWinDoc2020}, we assume that the nonlinear reaction term is Lipschitz continuous, preventing global exponential stability. However, we go deeper in the stability analysis, because we deal with the limit case of exponential stability as function of the boundary control gain.  

Exploiting observability and controllability properties of the linearized parabolic  equation (as introduced in the seminal paper \cite{sakawa1983feedback}), 
\cite{wang2025constructive} succeeds to solve the boundary stabilization problem of a semilinear parabolic equation by means of an observer-based dynamical controller. In this paper, inspired by \cite{lhachemi2025static} and \cite{YevZB2025}, we derive a static output feedback control candidate and a condition on the control gain to get the (local) exponential stability of the nonlinear closed-loop infinite-dimensional model. More specifically, by an appropriate control Lyapunov function, we state the relation between the speed of convergence (or divergence) with respect to the boundary control gain. Moreover, we show how the basin of attraction is related to the amplitude of the tuning parameter (see Theorem~\ref{prop1} below).

The paper is organized as follows. Section~\ref{sec:applications} contains a stability analysis of a nonlinear tubular reactor model. Section~\ref{sec:main_proof} collects the technical proofs of the main results and the statements of intermediate steps.
Section~\ref{sec:num} contains some numerical simulations illustrating our result on the reactor model together with a discussion on the choice of the control parameter. Some concluding remarks are given in 
Section~\ref{sec:con}.

\section{Main Results}\label{sec:applications}

We consider an isothermal mathematical model of an axially dispersed tubular reactor with a reaction of the type ``$A \to$ product''. The model is governed by a single nonlinear PDE with Danckwerts boundary conditions, where the control function enters through the boundary conditions. A similar, but coupled, model with two reactants was considered in \cite{AylAchLaa2007}, while non-isothermal models have been studied, for example, in \cite{HasWinDoc2020} and the references therein. For the extended survey of the models, we refer the reader to \cite{Doc2025}. The essential difference in our analysis lies in the treatment of the boundary control problem and in the study of the system behavior depending on the boundary control.

Consider the nonlinear parabolic equation with Robin-type (or Danckwerts) boundary conditions (see, e.g., \cite{NauMal1983,SH13}):
\begin{equation}\label{dftr}
\begin{aligned}
    &\frac{\partial C_A(x,t)}{\partial t}=D\frac{\partial^2C_A(x,t)}{\partial x^2} -v\frac{\partial C_A(x,t)}{\partial x}-kC_A(x,t)^n, 
\\ &C_A(0,t)=u(t)+\frac{D}{v}\frac{\partial C_A(x,t)}{\partial x}\Big|_{x=0}, \quad
\frac{\partial C_A(x,t)}{\partial x}\Big|_{x=l}=0,
\end{aligned}
\end{equation}
where $(x,t)\in[0,l]\times[0,\infty)$, $C_A(x,t)$ is the reactant $A$ concentration inside the reactor at distance $x$ from the inlet and at time $t$.
Here, $l>0$ denotes the length of the reactor tube, $n>0$ is the reaction order, $D>0$ is the axial dispersion coefficient, $v>0$ is the flow rate of the reaction stream, and $k>0$ is the reaction rate constant. The function $u(t)$  is the concentration of component $A$ in the inlet stream and represents the control input.

Let us first consider the steady-state problem:
\begin{equation}\label{eq:steady-state}
\begin{aligned}
    &D\overline{C}_A''(x) -v\overline{C}_A'(x) -k\overline{C}_A(x)^n=0, 
\qquad x\in[0,l],
\\ &\overline{C}_A(0)=\bar{u}+\tfrac{D}{v}\overline{C}_A'(0), \quad \overline{C}_A'(l)=0,
\end{aligned}
\end{equation}
where $\bar{u}=\text{const}>0$. From the economic and physical point of view, as well as by technical considerations, it is usually assumed that:
\begin{equation}\label{eq:constant_strip}
0\leqslant \overline{C}_A(x)\leqslant \bar{u}, \qquad x\in[0,l].
\end{equation}
Similar limitations will also be assumed for the function $C_A(x,t)$ later. Let us study the existence and uniqueness of the steady-state solution for the model.

\begin{lemma}\label{lem:steady-state}
Let $n>1$ and $\bar{u}\le\left(\frac{v^2}{4Dkn}\right)^\frac{1}{n-1}$. Then problem \eqref{eq:steady-state} has a unique solution $\overline{C}_A\in H^2(0,l)$ satisfying \eqref{eq:constant_strip}, where $H^2(0,l)$ denotes the corresponding Sobolev space.
\end{lemma}

Now, define the feedback control function $u$ by
\begin{equation}\label{control_1}
u(t)=\bar{u}+\alpha\,C_A(0,t), \qquad  \alpha\in\mathbb{R}.
\end{equation}

Consider the following initial condition for system \eqref{dftr}:
\begin{equation}\label{initial}
C_A(x,0)=C_0(x), \qquad  x\in [0,l],
\end{equation}
which satisfies the bound:
\begin{equation}\label{eq:initial_bound}
0\leqslant C_0(x)\leqslant M\varphi(x) \qquad\text{a.e. in }(0,l),
\end{equation}
for some $M>0$ and function $\varphi$ defined by
\begin{equation}\label{eq:phi}
\varphi(x)=-(x-l)^2+l^2+\tfrac{2Dl}{v(1-\alpha)}, \quad \text{for all }x\in[0,l].
\end{equation}
{ The choice of the specific constraint \eqref{eq:initial_bound} is justified in Remark~\ref{rem:invariance}.} Let us also introduce the following { critical} value:
\begin{equation}\label{eq:M*}
M^*:=\left(\tfrac{v^2}{4Dkn}\right)^\frac{1}{n-1}\tfrac{v(1-\alpha)}{l^2v(1-\alpha)+2Dl}.
\end{equation}

The following theorem is the main result of the paper and establishes the exponential convergence of the solutions to the closed-loop system \eqref{dftr}, \eqref{initial}, \eqref{control_1} towards the steady-state solution $\overline{C}_A$.
The theorem  employs the standard notion of mild solutions as defined, e.g., in \cite[Definition 5.1.4, Section 5]{CurZwa2020}).
\begin{theorem}\label{th2}
Let $n>1$, $0<M<M^*$, and let $\overline{C}_A$ be a steady-state solution defined by \eqref{eq:steady-state}. Then for any $\alpha<1$, the steady-state solution $\overline{C}_A(\cdot)$ of system \eqref{dftr}, \eqref{initial}, \eqref{control_1} is  exponentially stable. More precisely, there exist constants $K>0$ and $\omega>0$ such that for all initial conditions satisfying \eqref{eq:initial_bound}, the corresponding mild solution $C_A$ satisfies, for all $t\geqslant0$,
\begin{equation}\label{eq:exp_estimate_for_C}
\|C_A(\cdot,t)-\overline{C}_A\|_{L^2(0,l)}\leqslant K e^{-\omega\,t}\|C_0-\overline{C}_A\|_{L^2(0,l)},
\end{equation}
The precise value of $\omega$ is given in  Remark~\ref{rem3} below { and does not depend on $M$}.
\end{theorem}

\section{Proof of the Main Results}\label{sec:main_proof}

\subsection{System in deviations}\label{sec:system_in_deviations}

{ We first reformulate problem \eqref{dftr}, \eqref{initial}, \eqref{control_1} in terms of deviations from the steady-state solution $\overline{C}_A$ by introducing the function $\xi(t)= C_A(\cdot,t)-\overline{C}_A(\cdot)$ and rewriting the system in an abstract form:}
%
\begin{equation}\label{eq1.5}
\begin{aligned}
 \dot{{\xi}}(t)&=\mathcal{A}\,{\xi}(t)+r({\xi}(t)),\quad t\in[0,\infty),
 \\\xi(0)&=\xi_0\in X_0^M.   
\end{aligned}
\end{equation}
Here, the linear operator $\mathcal{A}:D(\mathcal{A})\to X$ is defined by
\begin{equation}\label{opA}
    \mathcal{A}:{\xi}\in D(\mathcal{A}) \mapsto \mathcal{A}{\xi} =
            (D\xi'' -v\xi')\in X,
\end{equation}
where $X=L^2(0,l)$ and, with some $\alpha\in\mathbb{R}$,
\begin{equation*}
D(\mathcal{A})=
\left\{\xi\in H^2(0,l)\,\middle|\,
\begin{array}{l}
(1-\alpha)\xi(0)=\frac{D}{v}\xi'(0),\\
\xi'(l)=0
\end{array}
\right\}.
\end{equation*}
The nonlinear operator $r:X_0^M\subset X\to X$ in~\eqref{eq1.5} is represented by 
\begin{equation}\label{eq3.6}
    r({\xi})=k\overline{C}_A^n-k (\xi+\overline{C}_A)^n, \qquad\xi\in X_0^M,
\end{equation}
where $X_0^M\subset X$ is the subset defined by
\begin{equation}\label{eq:X_0}
X_0^M = \{\xi\in X \,|\, 0\leqslant\xi(x)+\overline{C}_A(x)\leqslant M \varphi(x)\;\text{a.e.\,in}\,(0,l)\},
\end{equation}
with $M>0$ and $\varphi$ defined by \eqref{eq:phi}.
Observe that if $\alpha<1$, then $\varphi(x)>0$ for all $x\in[0,l]$. Consequently, $X_0^M$ contains an open neighborhood of zero with respect to the $L^\infty(0,l)$ norm. The analysis of trajectories constrained by $X_0^M$ is similar to the regional stability concept~\cite{zerrik2003regional}; however, the class of nonlinear parabolic models considered here was not addressed in the aforementioned paper.

We now state the following result for system \eqref{eq1.5}.

\begin{proposition}\label{prop1}
Let $n>1$, and $M<M^*$. Then for any $\alpha<1$, the  equilibrium $\xi= 0$ of the closed-loop system \eqref{eq1.5} is exponentially stable. More precisely, there exists a constant $K>0$ such that, for each $\xi_0\in X_0^M$, the corresponding solution $\xi(t)$ satisfies
\begin{equation}\label{eq3.8}
\|\xi(t)\|_X\leqslant K e^{\lambda_0(\alpha)t}\|\xi_0\|_X \quad\text{for all}\;t\geqslant0,
\end{equation}
where $\lambda_0(\alpha)=\sup_{\lambda\in\sigma(\mathcal{A})} \operatorname{Re}(\lambda)<0$.
\end{proposition}

It is clear that Theorem~\ref{th2} follows directly from Proposition~\ref{prop1}, and the remainder of this section is devoted to the proof of Proposition~\ref{prop1}. Subsections~\ref{sec:auxiliary} and \ref{sec:operator_A} collect the auxiliary results needed for the argument, while Subsection~\ref{sec:proof} provides the proof of the proposition.

\subsection{Well-posedness and invariance of $X_0^M$}\label{sec:auxiliary}

The following lemma establishes conditions for the well-posedness, nonnegativity, and boundedness of solutions to system \eqref{dftr}, \eqref{initial}, \eqref{control_1} under condition \eqref{eq:initial_bound}. It follows from the positive invariance of the set $X_0^M$, defined in \eqref{eq:X_0}, for the system \eqref{eq1.5}.

\begin{lemma}\label{lem:invariance}
Let $n>1$ and $\alpha<1$.  Then, for any $M>0$ and any $\xi_0$ in $X_0^M$, problem \eqref{eq1.5} admits a unique mild solution $\xi\in C([0,\infty);X_0^M)$.
\end{lemma}

\begin{proof}

{ It is easy to show that the nonlinearity $r$, defined by \eqref{eq3.6}, 
is quasi-decreasing in $\xi$ for $\xi+\overline{C}_A\ge 0$ 
and $n>1$ (see \cite{AreDan2025} for the definition of quasi-monotonicity). 
Moreover, in the case $n>1$, $\alpha<1$, and $M>0$, it is also Lipschitz 
continuous on $X_0^M$. Indeed, for all $\xi_1,\xi_2\in X_0^M$, we obtain}
\begin{equation*}
\begin{aligned}
\|r(\xi_1)-r(\xi_2)\|_X&\leqslant \sup_{\xi\in X_0^M}
kn\big|\overline{C}_A(x)+\xi(x)\big|^{n-1}\|\xi_1-\xi_2\|_X
\\&\leqslant kn\sup_{x\in[0,l]}\big(M\varphi(x)\big)^{n-1}
\|\xi_1-\xi_2\|_X \leqslant L_r\,\|\xi_1-\xi_2\|_X,
\end{aligned}
\end{equation*}
with $L_r=kn\,\Bigl(M\,\tfrac{l^2v(1-\alpha)+2Dl}{v(1-\alpha)}\Bigr)^{n-1}$. For a detailed analysis, we refer the reader to \cite{YevZB2025}. 

{ Having established quasi-monotonicity and Lipschitz continuity of $r$, 
we apply \cite[Theorem~1.2]{AreDan2025} to conclude that problem 
\eqref{eq1.5} admits a unique mild solution $\xi$ satisfying}
\begin{equation*}
\xi_{\min}(x)\leqslant \xi(x,t)\leqslant \xi_{\max}(x) 
\qquad \text{a.e. in }(0,l),\ \forall\, t\geqslant 0,
\end{equation*}
where $\xi_{\min}$ and $\xi_{\max}$ { form an ordered pair of weak sub- and super-solutions of a steady-state problem for \eqref{eq1.5} (see \cite[Definition~1.1]{AreDan2025}).}

It is easy to check that $\xi_{\min}(x)\equiv -\overline{C}_A(x)$ is a sub-solution of problem \eqref{eq1.5}, since $\overline{C}_A$ solves problem \eqref{eq:steady-state}. 

Now, we prove that $\xi_{\max}(x)=M\varphi(x)-\overline{C}_A(x)$ ($\varphi$ is defined by \eqref{eq:phi}) is a super-solution of problem \eqref{eq1.5} for any $M>0$. 
Note, that the function $\varphi(x)$ is increasing on $[0,l]$, since $\varphi'(x)=2(l-x)\geqslant0$ for all $x\in[0,l]$, and therefore $\varphi(x)\geqslant \varphi(0)=\frac{2Dl\,a}{v(1-\alpha)}$.
Furthermore, since $\overline{C}_A$ is a solution of \eqref{eq:steady-state}, it is easy to check that $\xi_{\max}\in D(\mathcal{A})$ and  $-\mathcal{A}\xi_{\max}-r(\xi_{\max})\ge 0$.
This proves that $\xi_{\max}$ is a super-solution of the problem.
\end{proof}

\begin{remark}\label{rem:invariance}
{ We note that the function $\varphi$ has been specifically chosen as a super-solution of the steady-state problem corresponding to~\eqref{eq1.5}. The parameter $M$ is then introduced to control the ''size'' of the set $X_0^M$.} In \cite{AylAchLaa2007}, a similar approach is used to prove the existence and uniqueness of solutions. Namely, the strip $0\le C_A \le \bar{C}=\mathrm{const}$ is introduced, and its invariance with respect to the linear semigroup is established. In contrast, the set $X_0^M$ introduced here is invariant with respect to the nonlinear semigroup generated by system \eqref{eq1.5}.
\end{remark}

\subsection{Spectrum of the operator $\mathcal{A}$}\label{sec:operator_A}

In the next two lemmas, we study the dependence of the spectrum of the operator $\mathcal{A}$ on $\alpha$.

\begin{lemma}\label{lem:spectrum_of_A}
The spectrum of the operator $\mathcal{A}$ defined by \eqref{opA} is real and discrete, and has the following form:
\begin{equation}\label{eq:eigenvalues}
\lambda_k(\alpha) = -\tfrac{v^2}{4D} - \theta_k(\alpha),
\end{equation}
where the values $\theta_k$, $k\in\mathbb{Z}^+$, depend on $\alpha$.
\end{lemma}

\begin{proof}
Consider the eigenvalue problem
\begin{equation}\label{eq:eigen_problem}
\mathcal{A}\xi=\lambda\xi,\qquad \lambda\in\mathbb{C},\ \xi\in D(\mathcal{A}).
\end{equation}
Introduce the transformation $y(x)=e^{-\frac{v}{2D}x}\xi(x)$. A direct computation yields $\mathcal{A}\xi = e^{\frac{v}{2D}x}\bigl(D y'' - \frac{v^2}{4D} y\bigr)$.
Hence, the eigenvalue problem \eqref{eq:eigen_problem} is equivalent to
$-D y'' = \theta y$, $y'(0)=\gamma\,y(0)$ 
and
$y'(l)=-\delta y(l)$,
where $\theta := -\Bigl(\lambda + \frac{v^2}{4D}\Bigr)$, $\gamma=\gamma(\alpha)=\frac{v}{D}\left(\frac12-\alpha\right)$, and $\delta=\frac{v}{2D}>0$.
Thus, we obtain the following Sturm--Liouville eigenvalue problem
\begin{equation}\label{eq:operator_L}
\mathcal{L}y:=- D y'' = \theta y, \qquad y\in D(\mathcal{L}),
\end{equation}
with the operator $\mathcal{L}:D(\mathcal{L})\subset L^2(0,l)\to L^2(0,l)$, where $$D(\mathcal{L})=\left\{y\in H^2(0,l):y'(0)=\gamma\,y(0), y'(l)=-\delta y(l)\right\}.$$ It can be easily verified that operator $\mathcal{L}$ is self-adjoint. Hence, the spectrum of $\mathcal{L}$ is discrete and consists of only real eigenvalues $\theta_k(\alpha)$, $k\in\mathbb{Z}^+$ (see, e.g., \cite[Section XII.2]{DunSch1988}). Therefore, the spectrum of $\mathcal{A}$ is represented by \eqref{eq:eigenvalues}.
\end{proof}

Without loss of generality, we assume that the sequence $\{\theta_k\}$, ${k\in\mathbb{Z}^+}$, is ordered increasingly. This assumption is justified in the proof of the following lemma. In this regard, we note that
$\lambda_0(\alpha)=\max_{k\in\mathbb{Z}^+}\lambda_k(\alpha)$.
We also introduce the critical value $\alpha^*$, which plays an important role in the subsequent spectral analysis:
\begin{equation}\label{eq:alpha*}
\alpha^*:=\tfrac{1}{2}+\tfrac{D}{vl+2D}\in\left(\tfrac{1}{2},1\right).   
\end{equation}

\begin{lemma}\label{lem:dependelne_on_alpha}
For the operator $\mathcal{A}$ defined by \eqref{opA}, the following statements hold:
\begin{itemize}
    \item[$(i)$] The principal eigenvalue $\lambda_0(\alpha)$ is increasing in $\alpha\in\mathbb R$.

    \item[$(ii)$] If $\alpha < \alpha^*$, then $\lambda_0(\alpha)<-\dfrac{v^2}{4D}$.

    \item[$(iii)$] If $\alpha = \alpha^*$, then $\lambda_0(\alpha)=-\dfrac{v^2}{4D}$.

    \item[$(iv)$] If $\alpha^*<\alpha<1$, then $-\dfrac{v^2}{4D}<\lambda_0(\alpha)<0$.
    
    \item[$(v)$] If $\alpha = 1$, then $\lambda_0(\alpha)=0$.
    
    \item[$(vi)$] If $\alpha > 1$, then $\lambda_0(\alpha)>0$.
\end{itemize}
\end{lemma}

\begin{proof}
Using the standard variational characterization of the principal eigenvalue of a self-adjoint Sturm--Liouville operator $\mathcal{L}$ by means of the Rayleigh quotient (see, e.g., \cite{Zet2005}), we get $\theta_0(\alpha)=\min_{0\neq y\in H^2(0,l)}\frac{\langle\mathcal{L}y,y\rangle}{\langle y,y\rangle}$, where $\langle\cdot,\cdot\rangle$ denotes the standard inner product in $L^2(0,l)$. This implies:
\begin{equation*}
\theta_0(\alpha)=\min_{\substack{y\in H^2(0,l)\\ y\neq 0}}\frac{D\int_0^l |y'|^2\,dx - D\delta |y(l)|^2 + D\gamma(\alpha)|y(0)|^2}
{\int_0^l |y|^2\,dx}.
\end{equation*}
We see that $\theta_0(\alpha)$ is increasing in $\gamma(\alpha)$, hence decreasing in $\alpha$.
Therefore, $\lambda_0(\alpha) = -\frac{v^2}{4D} - \theta_0(\alpha)$ is increasing in $\alpha$, which proves $(i)$.

Now, we focus on finding the eigenvalues of the operator $\mathcal{L}$, defined in \eqref{eq:operator_L}. Consider three cases.
\begin{enumerate}
    \item[1)] $\theta=0$, meaning that $\lambda=-\frac{v^2}{2D}$.
\end{enumerate}
Then, the solutions of \eqref{eq:operator_L} take the form
$y(x)=C_1x+C_2,$ 
for all $x \in (0,l)$,
for constants $C_1,\,C_2\in\mathbb{R}$,
and the boundary conditions imply:
$C_1-\gamma(\alpha)C_2=0$ and
$(1+\delta l)C_1+\delta C_2=0$.
The corresponding determinant condition gives $\gamma(\alpha)=-\frac{\delta}{1+\delta l}$ which implies $\alpha=\alpha^*$. Hence, the eigenvalue $\lambda=-\frac{v^2}{2D}$ belongs to the spectrum of the operator $\mathcal{A}$ only in the case $\alpha=\alpha^*$, where $\alpha^*$ is defined in \eqref{eq:alpha*}.

\begin{enumerate}
    \item[2)] $\theta<0$. Denote $\theta=-Dq^2$ with some parameter $q>0$.
\end{enumerate}
 In this case, the solutions of the eigenvalue problem \eqref{eq:operator_L} have the following form: 
$
y(x)=C_1e^{qx}+C_2e^{-qx}$, for all 
$x \in (0,l)$ with $C_1,\,C_2\in\mathbb{R}$,
and the boundary conditions yield:
\begin{equation}\label{eq:BC_nu<0}
(q-\gamma)C_1-(q+\gamma)C_2= 
(q+\delta)e^{2ql}C_1-(q-\delta)C_2=0.
\end{equation}
The corresponding determinant condition can be written as
\begin{equation}\label{eq:R(q)}
e^{2ql}=\frac{(q-\gamma)(q-\delta)}{(q+\gamma)(q+\delta)}=:R(q), \qquad q>0.
\end{equation}
At this stage, we assume that $q\neq -\gamma$ when $\gamma<0$, since this would imply
$\delta=-\gamma$, that is, $\alpha=1$. The case $\alpha=1$ will be treated separately below.

Now we analyze equation \eqref{eq:R(q)}. First, we note that $e^{2ql}>1$ since $q>0$. A standard analysis shows that $R(q)>1$ only if $\gamma<0$ $\left(\alpha>\frac{1}{2}\right)$ and $0<q<-\gamma$. Moreover, under the condition $\gamma<0$, the function $S(q):=R(q)-e^{2ql}$ satisfies $S(0)=0$ and $S(q)\to+\infty$ as $q\to-\gamma$. This means that a sufficient condition for the existence of roots $S(q)=0$ on the interval $q\in(0,-\gamma)$ is $S'(0)<0$, which leads to the condition $\alpha>\alpha^*$ with $\alpha^*$ from \eqref{eq:alpha*}.
%
%
This proves $(ii)$ and $(iii)$.

\begin{enumerate}
    \item[3)] $\theta>0$. Denote $\theta=Ds^2$ with some parameter $s>0$.
\end{enumerate}
 In this case, the solutions of the eigenvalue problem \eqref{eq:eigen_problem} have the following form:
$y(x)=C_1\cos(sx)+C_2\sin(sx)$,
for all $x \in (0,l)$. 
The determinant condition for the 
boundary conditions lead to 
\begin{equation}\label{eq:for_s}
(s^2-\delta\gamma(\alpha))\tan(sl)=s(\gamma(\alpha)+\delta),\qquad s>0.
\end{equation}

\begin{enumerate}
    \item[4)] $\alpha=1$. In this case, $\gamma=-\delta$.
\end{enumerate}
The boundary conditions \eqref{eq:BC_nu<0} translate into:
$
(q+\delta)C_1=0$, $(q+\delta)e^{ql}C_1-(q-\delta)e^{-ql}C_2=0$,
which leads to the only positive eigenvalue $\theta=-D\delta^2=-\frac{v^2}{4D}$ for operator $\mathcal{L}$, which proves $(v)$.

Finally, statements $(iv)$ and $(vi)$ follow from the others.
\end{proof}





\subsection{Proof of Lemma \ref{lem:steady-state}}\label{sec:steady-state}
We rewrite problem \eqref{eq:steady-state} in an abstract form for $y=C_A-\bar{u}$:
\begin{equation}\label{eq:steady-state_abstract}
 \mathcal{B}\,{y}(t)=f(y).
\end{equation}
Here, the linear operator $\mathcal{B}=\mathcal{A}|_{\alpha=0}$ with $\mathcal{A}$ defined by \eqref{opA}, and the nonlinear operator $f:Y_0^{\bar{u}}\subset X\to X$ is given by 
\begin{equation*}
\begin{aligned}
f({y})&=k (y+\bar{u})^n, \qquad y\in Y_0^{\bar{u}},
\\Y_0^{\bar{u}} &= \{y\in X \,|\, 0\leqslant y(x)+\bar{u}\leqslant \bar{u}\;\text{ a.e. in }(0,l)\}.
\end{aligned}
\end{equation*}
We note that by \cite[Section~4.1]{CurZwa2020}, operator $\mathcal{B}$ is invertible, and, by Lemma~\ref{lem:dependelne_on_alpha}, it holds 
$\|\mathcal{B}^{-1}\|\leq\frac{1}{-\lambda_0(0)}<\frac{4D}{v^2}.$ 
Following the same arguments as in Lemma~\ref{lem:invariance}, we can establish the Lipschitz condition for the nonlinear operator $f$ with the Lipschitz constant $L_f=kn\bar{u}^{n-1}$.

Any solution of \eqref{eq:steady-state_abstract} can be represented as a fixed point of operator $\Phi(y):=\mathcal{B}^{-1}f(y)$. For all $y_1,y_2\in X$, we have
$$
\|\Phi(y_1)-\Phi(y_2)\|_X\leq\|\mathcal{B}^{-1}\|\,L_f\,\|y_1-y_2\|_X.
$$
By the Banach fixed-point theorem, $\Phi$ has a unique fixed point, if $\|\mathcal{B}^{-1}\|\,L_f<1$, that is $\bar{u}\le\left(\frac{v^2}{4Dkn}\right)^\frac{1}{n-1}$.

\subsection{Proof of Proposition \ref{prop1}}\label{sec:proof}

Since $M<M^*$, we get $M\varphi(x)<\left(\frac{v^2}{4Dkn}\right)^\frac{1}{n-1}$. This fulfills the requirement of Lemma~\ref{lem:steady-state} and ensures the existence and uniqueness of the steady-state solution $\overline{C}_A(x)$.

Introduce the following multiplication operator $P:X\to X$:
\begin{equation}\label{eq:operator_P}
(P\xi)(x) = e^{-\frac{v}{D}x}\xi(x),\qquad x\in[0,l].
\end{equation}
Recall that the inner product in $X=L^2(0,l)$ is defined by $\langle \xi,\eta\rangle_X=\int_0^l \xi(x)\eta(x)\,dx$ for any $\xi,\eta\in X$. Define also the weighted inner product w.r.t. operator $P$:
\begin{equation*}
\langle \xi, \eta\rangle_P := \langle P\xi, \eta\rangle_{X} = \int_0^l e^{-\frac{v}{D}x}\xi(x)\eta(x)\,dx.
\end{equation*}
Now we prove that $\mathcal{A}$ is self-adjoint with respect to $\langle\cdot,\cdot\rangle_P$. Indeed, for any $\xi,\eta\in D(\mathcal{A})$, we have
\begin{equation*}
\begin{aligned}
\langle \mathcal{A}\xi,\eta\rangle_P &= \int_0^l (D\xi'' -v\xi' )e^{-\frac{v}{D}x}\eta \,dx
\\&=\left[D\xi'  e^{-\frac{v}{D}x}\eta \right]_0^l 
-\int_0^l D\xi' \left(-\frac{v}{D}e^{-\frac{v}{D}x}\eta  + e^{-\frac{v}{D}x}\eta' \right)dx 
-\int_0^l v\xi' e^{-\frac{v}{D}x}\eta \,dx
\\&=-D\xi'(0)\eta(0)-\int_0^l D\xi' e^{-\frac{v}{D}x}\eta' dx
\\&=-v(1-\alpha)\xi(0)\eta(0)
-\left[D\xi e^{-\frac{v}{D}x}\eta' \right]_0^l
+\int_0^l D\xi \left(-\frac{v}{D}e^{-\frac{v}{D}x}\eta'  + e^{-\frac{v}{D}x}\eta'' \right)dx
\\&=-v(1-\alpha)\xi(0)\eta(0)+D\xi(0)\eta'(0)
+\int_0^l e^{-\frac{v}{D}x}\xi \left(-v\eta'  +D\eta'' \right)dx
=\langle \xi,\mathcal{A}\eta\rangle_P
\end{aligned}
\end{equation*}
%
%
By \cite{DelDocWin2003}, the operator $-\mathcal{A}$ is a Sturm--Liouville operator and, consequently, a Riesz-spectral operator. Denote by $\{\phi_k\}_{k=0}^{\infty}$ and $\{\psi_k\}_{k=0}^{\infty}$ the sets of eigenfunctions of operator $\mathcal{A}$ and $\mathcal{A}^*$ respectively. Now any $\xi \in D(\mathcal{A})$ can be expanded in the following way:
\begin{equation*}
\xi = \sum_{k=0}^{\infty} c_k\,\phi_k, \qquad \text{where } c_k:=\langle \xi, \psi_k\rangle_{X}.
\end{equation*}
Moreover, since $\mathcal{A}$ is self-adjoint w.r.t. $\langle\cdot,\cdot\rangle_P$, its eigenfunctions are orthogonal in this inner product. Indeed,
\begin{equation*}
\lambda_k\langle\phi_k,\phi_m\rangle_P = \langle \mathcal{A}\phi_k,\phi_m\rangle_P = \langle\phi_k,\mathcal{A}\phi_m\rangle_P = \lambda_m\langle\phi_k,\phi_m\rangle_P.
\end{equation*}
This implies $(\lambda_k - \lambda_m)\langle\phi_k,\phi_m\rangle_P = 0$,
which leads to $\langle \phi_k, \phi_m\rangle_P = 0$ for  $k\neq m$.
Computing $\langle \mathcal{A}\xi, P\xi\rangle_{X}$, we obtain:
\begin{equation}\label{eq:aux_est_for_Lyapunov}
\begin{aligned}
\langle \mathcal{A}\xi, P\xi\rangle_{X}
&= \langle \mathcal{A}\xi, \xi\rangle_P 
=\biggr\langle \mathcal{A}\sum_{k=0}^\infty c_k\phi_k,\ \sum_{m=0}^\infty c_m\phi_m\biggr\rangle_P 
= \biggr\langle \sum_{k=0}^\infty c_k\lambda_k\phi_k,\ \sum_{m=0}^\infty c_m\phi_m\biggr\rangle_P
\\&= \sum_{k,m=0}^\infty c_k c_m \lambda_k \langle\phi_k,\phi_m\rangle_P 
\leq \lambda_0\sum_{k=0}^{\infty} c_k^2\|\phi_k\|_P^2 
=\lambda_0\langle\xi,\xi\rangle_P = \lambda_0\langle P\xi,\xi\rangle_{X}.
\end{aligned}
\end{equation}
Introduce the Lyapunov functional $V(\xi)=\langle P\xi,\xi\rangle_{X}$, $\xi\in X$, with the operator $P$ defined by \eqref{eq:operator_P}. Differentiating $V$ with respect to time along the solutions of the nonlinear system~\eqref{eq1.5}, and using inequality \eqref{eq:aux_est_for_Lyapunov}, we get:
\begin{equation}\label{eq:Lyapunov_derivative}
\begin{aligned}
\dot{V}(\xi)&
\leq 2\lambda_0\langle P\xi,\xi\rangle_{X}+2\langle\xi,r(\xi)\rangle_P.
\end{aligned}
\end{equation}
For $r$ defined by \eqref{eq3.6}, using the fact that $(a-b)(a^n-b^n)\ge0$ for any positive numbers $a,b,n$, we obtain for any $\xi\in X_0^M$
\begin{equation*}
\begin{aligned}
\langle\xi,r(\xi)\rangle_P&=k\int_0^l e^{-\frac{v}{D}x}\xi\left(\overline{C}_A^n- (\xi+\overline{C}_A)^n\right)dx
\\&=-k\int_0^le^{-\frac{v}{D}x}\left(\overline{C}_A- (\xi+\overline{C}_A)\right)\left(\overline{C}_A^n- (\xi+\overline{C}_A)^n\right)dx\le0.
\end{aligned}
\end{equation*}
Thus, we derive from \eqref{eq:Lyapunov_derivative}: 
\begin{equation*}
\dot{V}(\xi)\le 2\lambda_0 V(\xi),
\qquad \xi\in D(\mathcal{A})\cap X_0^M,
\end{equation*}
which leads to the following exponential estimate:
\begin{equation*}
V(\xi(t))\le \langle P\xi_0,\xi_0\rangle_{X} e^{2\lambda_0 t},
\qquad \xi_0\in D(\mathcal{A})\cap X_0.
\end{equation*}
Since $e^{-\frac{vl}{D}}\leq e^{-\frac{v}{D}x} \leq 1$ for all $x\in[0,l]$, it is obvious, that $e^{-\frac{vl}{D}} \|\xi\|_X^2 \le \langle P\xi,\xi\rangle_{X} \le \|\xi\|_X^2$ for all $\xi\in X$.
Finally, we get
\begin{equation*}
\|\xi(t)\|_X
\le K e^{\lambda_0t}\|\xi_0\|_X
\quad \text{for all}\; \xi_0\in X_0,\; t\ge 0.
\end{equation*}
This leads to \eqref{eq3.8} with constant $K=e^{\frac{vl}{2D}}$.

\begin{remark}\label{rem3}
The precise value of decay rate $\omega$ from exponential estimate \eqref{eq:exp_estimate_for_C} is given by:
\begin{equation*}
    \omega=-\lambda_0(\alpha)=\frac{v^2}{4D}+\theta,\quad \theta=
    \begin{cases}
    -Dq^2, &\text{if }\alpha < \alpha^*,\\
    0, &\text{if }\alpha = \alpha^*,\\
    Ds^2, &\text{if }\alpha^*<\alpha <1,
    \end{cases}
\end{equation*}
where $\alpha^*$ is defined in \eqref{eq:alpha*}, $q>0$ is the maximal positive root of equation \eqref{eq:R(q)}, and $s>0$ is the minimal positive root of equation \eqref{eq:for_s}.
\end{remark}

\section{Simulation results}\label{sec:num}

For the numerical simulations, we choose the following parameters of the control system~\eqref{eq1.5}:
$k=0.001\,\tfrac{1}{s\cdot mol}$, $v=0.01\,\tfrac{m}{s}$,
$l=1\,m$, $D=0.0025\,\tfrac{m^2}{s}$, $n=2$.
To analyze the dynamical behavior for different choices of the control gain $\alpha$ in~\eqref{control_1} within the range { $\alpha<1$}, we define the initial function $\xi_0$ at $t=0$ as
\begin{equation}\label{w0fun}
\begin{aligned}
\xi_0(x)&=\mu M^* \varphi(x)= \mu M^*\left(l^2-(x-l)^2+\tfrac{2Dl}{v(1-\alpha)}\right),
\end{aligned}
\end{equation}
{ with some $\mu\in(0,1)$, $M^*$ defined by \eqref{eq:M*}, and $\varphi$ defined by \eqref{eq:phi}}. 
It is easy to verify that the chosen initial function $\xi_0$ satisfies the boundary conditions in $D(\mathcal A)$.
{ In the subsequent simulations, we choose  $\mu=0.9$ to guarantee that $\xi_0\in X_0^M$.}

Numerical solutions of system~\eqref{eq1.5} are computed in \texttt{MATLAB} R2025a using the \texttt{pdepe} function.
A preliminary step in this process requires defining the steady-state profile $\bar C_A(x)$ for each considered value of $\alpha$,
which is implemented via the numerical solution of the corresponding time-invariant boundary-value problem using \texttt{bvp4c}.

The resulting time-plots of the spatial $L^2$-norms of the numerical solutions are presented in Fig.~1 for different values of $\alpha$ within the range from
$\alpha_{\min}=-10$ to $\alpha_{\max}=0.95$. For convenience, the norms are rescaled to have the common initial value $1$ for all functions $\|\xi(t)\|_{L^2(0,1)}/\|\xi_0\|_{L^2(0,1)}$ at $t=0$.
These simulations clearly indicate the exponential decay of all solutions for large values of $t$.

\begin{figure}[h]\label{decay_multipleplots}
\begin{center}
\includegraphics[width=0.8\linewidth]{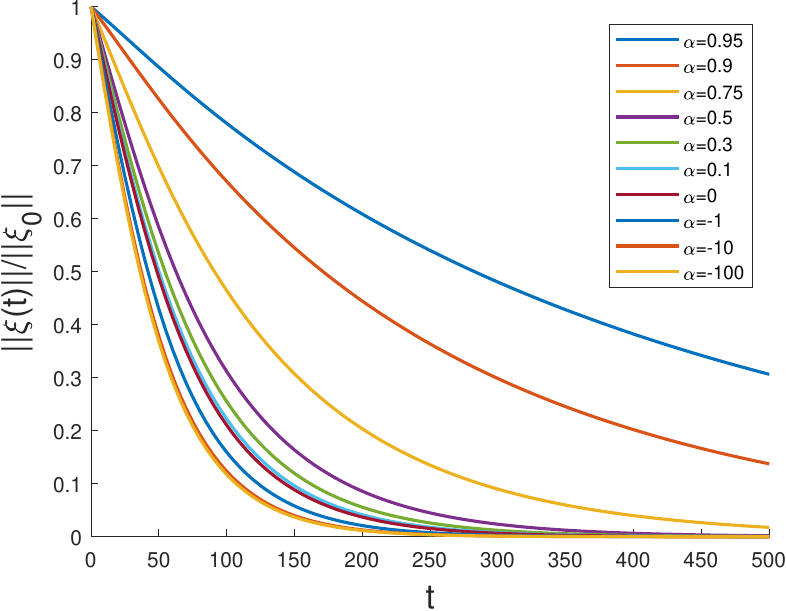}
\end{center}
\vskip-1ex
\caption{Norms of the solutions for different values of $\alpha$.}
\end{figure}

\begin{figure}[h]\label{decay_multipleplots}
\begin{center}
\includegraphics[width=0.8\linewidth]{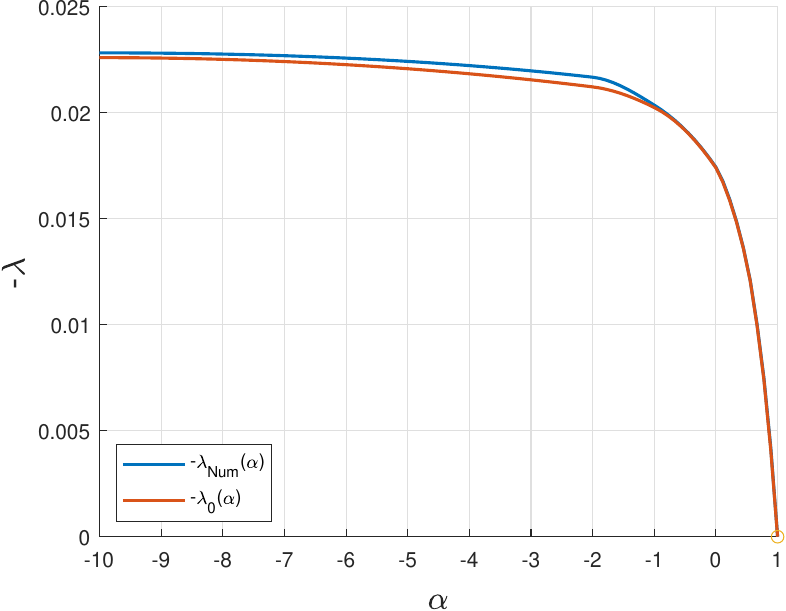}
\end{center}
\vskip-1ex
\caption{Norms of the solutions for different values of $\alpha$.}
\end{figure}

We also summarize the decay rate estimates for the considered closed-loop system in Table~\ref{tab1} for the indicated choices of $\alpha$,
where $\lambda_0$ is the principal eigenvalue of the linear operator $\mathcal A${, calculated by the formulas provided in Remark~\ref{rem3},} and $\lambda_{Num}$ is the numerical estimate of the Lyapunov exponent $\limsup_{t+\infty} \left( \frac{\ln\|\xi(t)\|_{L^2(0,1)}}{t} \right)$ for the solution $\xi(t)$ corresponding to the initial data~\eqref{w0fun}.
These results show good agreement between $\lambda_{Num}$ and $\lambda_0$.

\begin{table}[h!]
\centering
{\scriptsize
\begin{tabular}{|c|c|c|c|c|c|c|}
\hline
 & $\alpha=-10$ & $\alpha=-1$ & $\alpha=0$ & $\alpha=0.5$ & $\alpha=0.75$ & $\alpha=0.9$ \\ 
\hline
$-\lambda_{Num}$ & 0.0228
& 0.0204
 & 0.0174
 & 0.0129
 & 0.0082
 & 0.0038
 \\ 
\hline
$-\lambda_0$ & 0.0226
 & 0.0202
 & 0.0174
 & 0.0129
 & 0.0081
 & 0.0037
 \\ 
 \hline
\end{tabular}}
\caption{Decay rate estimates.}\label{tab1}
\end{table}

Fig.~2 represents the plots of $-\lambda_{Num}(\alpha)$ and $-\lambda_0(\alpha)$. 
As established in Lemma~\ref{lem:dependelne_on_alpha}, $-\lambda_0(\alpha)$ is strictly decreasing on $\alpha\in (-\infty,1)$
and vanishes at $\alpha=1$, and similar behavior is also observed for $-\lambda_{Num}(\alpha)$. 

It should be emphasized that there is {\em no optimal decay rate} in the sense that none of the considered decay characteristics ($\lambda_0(\alpha)$, $\lambda_{Num}(\alpha)$) has a global minimum over the range of the tuning parameter $\alpha\in (-\infty,1)$. Therefore, {\em any decrease in $\alpha$ improves the convergence rate}. However, our analysis shows that all these decay characteristics are bounded, with $\lambda_{Num}(\alpha)$ approaching its infimum $\lambda_0^*\approx -0.0234$ as $\alpha\to-\infty$.

\section*{Conclusion}\label{sec:con}

A key outcome of our study is the theoretical and computational justification of the impact of the gain parameter $\alpha$ in the boundary control of the parabolic nonlinear reaction model, which can play either a {\em stabilizing role} (for $\alpha<1$) or a {\em destabilizing role} (for $\alpha>1$).
The theoretical decay estimates obtained in Theorem~\ref{th2} and Proposition~\ref{prop1}, together  with the performed numerical analysis, confirm that the principal eigenvalue of the linearized system captures the asymptotic behavior of the considered nonlinear system in the sense of the exponential decay rate.

An important feature of our study is the consideration of the closed-loop system in an invariant set that contains an open neighborhood in the $L^\infty$-topology, but does not contain any neighborhood in the $L^2$-space. This framework differs from classical analysis in invariant sets defined via level sets of a Lyapunov functional. 
The study of the system behavior at $\alpha=1$, where the linear system has a zero principal eigenvalue, is also considered as a topic for further investigation for the nonlinear system under different reaction orders $n$, possibly integrating the theory of critical cases in the sense of Lyapunov with the theory of infinite-dimensional dynamical systems.

For realistic applications, future work will focus on developing a spectral analysis framework for a more complex model consisting of coupled parabolic PDEs governing the mass and energy balances  (cf.~\cite{SH13}). The nonlinearities arising from the Arrhenius law are expected to introduce additional theoretical and computational challenges, providing a promising direction for future research.

\end{document}